\documentclass[11pt]{amsart}

\usepackage{subcaption}
\usepackage{subfiles}
\usepackage{comment}
\usepackage{float}
\usepackage{marginnote}
\usepackage{tabu}
\usepackage{euscript}
\usepackage{mathdots}
\usepackage[dvipsnames]{xcolor}
\usepackage{graphicx}			
\usepackage{amssymb}
\usepackage{mathrsfs}
\usepackage{amsmath}
\usepackage{amsthm}
\usepackage{stmaryrd}
\usepackage{tikz}
\usepackage{tikz-cd}
\usepackage{accents}
\usepackage{upgreek}
\usepackage{enumerate}
\usepackage{bm}
\usepackage{mathtools}
\usetikzlibrary{patterns}
\usepackage[all]{xy}
\usepackage{caption}
\usepackage{url}
\usepackage{float}
\usepackage{todonotes} 
\usepackage{colonequals}
\usepackage{bbm}
\usepackage{longtable}
\usepackage[full]{textcomp}
\usepackage[cal=cm]{mathalfa}
\usepackage{xparse}
\usepackage{comment}
\usepackage{cite}
\usepackage{enumitem}

\makeatletter
\DeclareRobustCommand{\cev}[1]{%
  \mathpalette\do@cev{#1}%
}
\newcommand{\do@cev}[2]{%
  \fix@cev{#1}{+}%
  \reflectbox{$\m@th#1\vec{\reflectbox{$\fix@cev{#1}{-}\m@th#1#2\fix@cev{#1}{+}$}}$}%
  \fix@cev{#1}{-}%
}
\newcommand{\fix@cev}[2]{%
  \ifx#1\displaystyle
    \mkern#23mu
  \else
    \ifx#1\textstyle
      \mkern#23mu
    \else
      \ifx#1\scriptstyle
        \mkern#22mu
      \else
        \mkern#22mu
      \fi
    \fi
  \fi
}

\usetikzlibrary{calc}
\usetikzlibrary{fadings}
\usetikzlibrary{decorations.pathmorphing}
\usetikzlibrary{decorations.pathreplacing}
\usepackage{tikz,tikz-cd,tikz-3dplot}
\usepackage{pgfplots}
\usetikzlibrary{arrows,shadows,positioning, calc, decorations.markings, 
	hobby,quotes,angles,decorations.pathreplacing,intersections,shapes}
\usepgflibrary{shapes.geometric}
\usetikzlibrary{fillbetween,backgrounds}

\usepackage[margin=0.8in]{geometry}

\tikzset{
	commutative diagrams/.cd, 
	arrow style=tikz, 
	diagrams={>=stealth}
}
\tikzset{
	arrow/.pic={\path[tips,every arrow/.try,->,>=#1] (0,0) -- +(0,4pt);},
	pics/arrow/.default={triangle 90}
}
\tikzset{->-/.style={decoration={
			markings,
			mark=at position .6 with {\arrow{latex}}},postaction={decorate}}
}
\tikzset{
	c/.style={every coordinate/.try}
}

\theoremstyle{theorem}
\newtheorem{innercustomthm}{Theorem}
\newenvironment{customthm}[1]
{\renewcommand\theinnercustomthm{#1}\innercustomthm}
{\endinnercustomthm}

\theoremstyle{theorem}

\theoremstyle{theorem}

\makeatletter
\def\@tocline#1#2#3#4#5#6#7{\relax
	\ifnum #1>\c@tocdepth 
	\else
	\par \addpenalty\@secpenalty\addvspace{#2}%
	\begingroup \hyphenpenalty\@M
	\@ifempty{#4}{%
		\@tempdima\csname r@tocindent\number#1\endcsname\relax
	}{%
		\@tempdima#4\relax
	}%
	\parindent\z@ \leftskip#3\relax \advance\leftskip\@tempdima\relax
	\rightskip\@pnumwidth plus4em \parfillskip-\@pnumwidth
	#5\leavevmode\hskip-\@tempdima
	\ifcase #1
	\or\or \hskip 1em \or \hskip 2em \else \hskip 3em \fi%
	#6\nobreak\relax
	\dotfill\hbox to\@pnumwidth{\@tocpagenum{#7}}\par
	\nobreak
	\endgroup
	\fi}
\makeatother

\newcounter{marginnote}
\DeclareMathAlphabet{\mathpzc}{OT1}{pzc}{m}{it}

\usepackage[backref=page]{hyperref}
\hypersetup{
  colorlinks   = true,          
  urlcolor     = PineGreen,          
  linkcolor    = PineGreen,          
  citecolor   = PineGreen             
}

\usepackage{cleveref}

\theoremstyle{theorem}
\newtheorem{theorem}{Theorem}[section]

\newtheorem{corollary}[theorem]{Corollary}
\newtheorem{lemma}[theorem]{Lemma}
\newtheorem{proposition}[theorem]{Proposition}
\newtheorem*{problem*}{Problem}

\theoremstyle{definition}

\newtheorem{remark}[theorem]{Remark}

\newtheorem*{runningexample*}{Running example}

\newtheorem*{aside*}{Aside}

\newtheorem{proposition-definition}[theorem]{Proposition-Definition}

\newcommand{\xdashleftrightarrow}[2][]{\ext@arrow 3359\leftrightarrowfill@@{#1}{#2}}

\newcommand{\ol}[1]{\overline{#1}}

\newcommand{\bcd}{\begin{center}\begin{tikzcd}}
		\newcommand{\ecd}{\end{tikzcd}\end{center}}

\newcommand{\Aaff}{\mathbb{A}}

\newcommand{\OO}{\mathcal{O}}

\newcommand{\Z}{\mathbb{Z}}
\newcommand{\Q}{\mathbb{Q}}

\newcommand{\orb}{\mathrm{orb}}

\newcommand{\Mcal}{\mathcal{M}}

\newcommand{\Ical}{\mathcal{I}}

\makeatletter
\let\save@mathaccent\mathaccent
\newcommand*\if@single[3]{%
  \setbox0\hbox{${\mathaccent"0362{#1}}^H$}%
  \setbox2\hbox{${\mathaccent"0362{\kern0pt#1}}^H$}%
  \ifdim\ht0=\ht2 #3\else #2\fi
  }
\newcommand*\rel@kern[1]{\kern#1\dimexpr\macc@kerna}
\newcommand*\widebar[1]{\@ifnextchar^{{\wide@bar{#1}{0}}}{\wide@bar{#1}{1}}}
\newcommand*\wide@bar[2]{\if@single{#1}{\wide@bar@{#1}{#2}{1}}{\wide@bar@{#1}{#2}{2}}}
\newcommand*\wide@bar@[3]{%
  \begingroup
  \def\mathaccent##1##2{%
    \let\mathaccent\save@mathaccent
    \if#32 \let\macc@nucleus\first@char \fi
    \setbox\z@\hbox{$\macc@style{\macc@nucleus}_{}$}%
    \setbox\tw@\hbox{$\macc@style{\macc@nucleus}{}_{}$}%
    \dimen@\wd\tw@
    \advance\dimen@-\wd\z@
    \divide\dimen@ 3
    \@tempdima\wd\tw@
    \advance\@tempdima-\scriptspace
    \divide\@tempdima 10
    \advance\dimen@-\@tempdima
    \ifdim\dimen@>\z@ \dimen@0pt\fi
    \rel@kern{0.6}\kern-\dimen@
    \if#31
      \overline{\rel@kern{-0.6}\kern\dimen@\macc@nucleus\rel@kern{0.4}\kern\dimen@}%
      \advance\dimen@0.4\dimexpr\macc@kerna
      \let\final@kern#2%
      \ifdim\dimen@<\z@ \let\final@kern1\fi
      \if\final@kern1 \kern-\dimen@\fi
    \else
      \overline{\rel@kern{-0.6}\kern\dimen@#1}%
    \fi
  }%
  \macc@depth\@ne
  \let\math@bgroup\@empty \let\math@egroup\macc@set@skewchar
  \mathsurround\z@ \frozen@everymath{\mathgroup\macc@group\relax}%
  \macc@set@skewchar\relax
  \let\mathaccentV\macc@nested@a
  \if#31
    \macc@nested@a\relax111{#1}%
  \else
    \def\gobble@till@marker##1\endmarker{}%
    \futurelet\first@char\gobble@till@marker#1\endmarker
    \ifcat\noexpand\first@char A\else
      \def\first@char{}%
    \fi
    \macc@nested@a\relax111{\first@char}%
  \fi
  \endgroup
}
\makeatother

\newcommand{\GL}{\operatorname{GL}}

\newcommand{\DR}{\operatorname{DR}}
\newcommand{\Jac}{\operatorname{Jac}}
\newcommand{\aj}{\operatorname{aj}}

\crefname{equation}{eq.}{eqs.}
\crefname{eqnarray}{eq.}{eqs.}
\crefname{conjecture}{conjecture}{conjectures}
\crefname{lemma}{lemma}{lemmas}
\crefname{theorem}{theorem}{theorems}
\crefname{claim}{claim}{claims}
\crefname{remark}{remark}{remarks}
\crefname{proposition}{proposition}{propositions}
\crefname{section}{section}{sections}
\crefname{appendix}{appendix}{appendices}
\crefname{corollary}{corollary}{corollaries}
\crefname{figure}{figure}{figures}
\crefname{table}{table}{tables}
\crefname{example}{example}{examples}
\crefname{assumption}{assumption}{assumptions}
\crefname{definition}{definition}{definitions}
\crefname{innercustomthm}{theorem}{theorems}
\crefname{innercustomconj}{conjecture}{conjectures}

\setlist[enumerate,1]{label=(\roman*),itemsep=0.9ex}
\setlist[itemize]{itemsep=0.9ex}

\begin{document}

\title[Euler characteristics of higher rank DR loci in genus one]{Euler characteristics of higher rank \\ double ramification loci in genus one}
\author{Luca Battistella and Navid Nabijou}

\begin{abstract} 
Double ramification loci parametrise marked curves where a weighted sum of the markings is linearly trivial; higher-rank loci are obtained by imposing several such conditions simultaneously. We obtain closed formulae for the orbifold Euler characteristics of double ramification loci, and their higher-rank generalisations, in genus one. The rank-one formula is a polynomial, while the higher-rank formula involves greatest common divisors of matrix minors. The proof is based on a recurrence relation, which allows for induction on the rank and number of markings.
\end{abstract}

\maketitle

\vspace{-1cm}

\section*{Introduction}

\noindent Fix $g \geqslant 0$ and a vector $a=(a_1,\ldots,a_n) \in \Z^n$ with $\Sigma_{i=1}^n a_i=0$. The associated (open) double ramification locus is given on the level of closed points by:
\[ \DR_{g,n}(a) = \left\{ (C,p_1,\ldots,p_n) : \OO_C(\Sigma_{i=1}^n a_i p_i) \cong \OO_C \right\} \subseteq \Mcal_{g,n}.\]
Higher-rank double ramification loci are produced by intersecting several of the above loci. We determine the orbifold Euler characteristics of double ramification loci and their higher-rank generalisations in the first nontrivial genus, namely $g=1$.

Recently, strata of differentials have attracted considerable attention due to their position at the interface of dynamics and algebraic geometry \cite{EM,EMM,Filip,EFW}. Yet, little is known about their global topology \cite{KZ,CMZ}. In genus one, double ramification loci exhaust the strata of meromorphic $k$-differentials, due to the triviality of the canonical bundle.

\subsection{Results} We begin in \Cref{sec: rank 1} with the classical (rank-one) case. The main result is:

\begin{customthm}{X}[\Cref{thm: Euler char rank one}] \label{thm: rank one introduction} Given $a=(a_1,\ldots,a_n)$ the orbifold Euler characteristic of $\DR_{1,n}(a)$ is given by:
\[ \chi_\orb(\DR_{1,n}(a)) = \dfrac{(-1)^{n-1}(n-1)!}{24} \left( \sum_{i=1}^n a_i^2 - 2 \right).\]
\end{customthm}
\noindent The combinatorial prefactor equals $-\chi_\orb(\Mcal_{1,n})/2$ (see \Cref{lem: HarerZagier}). The above formula is obtained independently (and with a different proof) in the upcoming \cite{CMS}.

In \Cref{sec: higher rank} we proceed to the higher-rank case. Here the input data is an $r \times n$ matrix $A$ such that each row sums to zero. The associated higher-rank double ramificaton locus
\[ \DR^r_{g,n}(A) \subseteq \Mcal_{g,n} \]
is the intersection of the $r$ double ramification loci associated to the rows of $A$. We determine its orbifold Euler characteristic.

\begin{customthm}{Y}[\Cref{thm: higher rank DR}] \label{thm: higher rank introduction} The orbifold Euler characteristic of the higher-rank double ramification locus $\DR^r_{1,n}(A)$ is given by:
\[ \chi_\orb(\DR^r_{1,n}(A)) = \dfrac{(-1)^n}{12} \sum_{k=0}^r \sum_{\substack{\Ical \, \vdash [n] \\ \ell(\Ical)=k+1}} (-1)^k (\# I_1-1)! \cdots (\# I_{k+1}-1)! \cdot G_{k \times k}(A_\Ical)^2.\]
The sum is over partitions $\Ical=\{I_1,\ldots,I_{k+1}\}$ of the set $[n]=\{1,\ldots,n\}$, the contraction matrix $A_\Ical$ is the $r \times (k+1)$ matrix obtained by summing the columns of $A$ associated to each part of $\Ical$, and $G_{k \times k}(A_\Ical)$ denotes the greatest common divisor of all the $k \times k$ minors of $A_\Ical$.
\end{customthm}

We note in particular that in higher rank, the orbifold Euler characteristic is not polynomial in the entries of the matrix $A$.

In \Cref{sec: leading term} we simplify the $k=r$ term in the above sum, expressing it in terms of the $r \times r$ minors of the original matrix~$A$ (\Cref{prop: leading term}). When $r=1$ the formulae in Theorems \ref{thm: rank one introduction} and \ref{thm: higher rank introduction} are superficially different, but in \Cref{sec: comparing theorems X and Y} we identify them using symmetric function theory.

The following example families convey the flavour of \Cref{thm: higher rank introduction}:
\begin{align*} \chi_\orb \DR_{1,3}^2 \! \begin{bmatrix} a_1 & a_2 & a_3 \\ b_1 & b_2 & b_3 \end{bmatrix} & = \dfrac{1}{12} \big( \!-\!2 + \gcd(a_1,b_1)^2 + \gcd(a_2,b_2)^2 + \gcd(a_3,b_3)^2 - (a_1 b_2 - a_2 b_1)^2 \big), \\[0.3cm]
\chi_\orb \DR_{1,4}^2 \! \begin{bmatrix} a & - a & 0 & 0 \\ 0 & 0 & b & -b \end{bmatrix} & = \dfrac{1}{12} \big( 6 - 4a^2 - 4b^2 - 2\gcd(a,b)^2 + 4(ab)^2 \big), \\[0.3cm]
\chi_\orb \DR_{1,4}^3 \begin{bmatrix} a & -a & 0 & 0 \\ 0 & b & -b & 0 \\ 0 & 0 & c & -c \end{bmatrix} & = \dfrac{1}{12} \big( 6 - 2a^2 - b^2 - 2c^2 - 2\gcd(a,b)^2 - \gcd(a,c)^2 - 2\gcd(b,c)^2 \\
& \qquad \qquad - \gcd(a,b,c)^2 + (ab)^2 + (ac)^2 + (bc)^2 + 3\gcd(ab,ac,bc)^2 - (abc)^2 \big).
\end{align*}

\subsection{Proof strategy} The proof hinges on a recurrence relation, which we illustrate in rank one. Fix $a=(a_1,\ldots,a_n,a_{n+1})$ with $\Sigma_{i=1}^{n+1} a_i = 0$ and suppose without loss of generality that $a_{n+1} \neq 0$. Consider the forgetful morphism:
\begin{equation} \label{eqn: forgetful map introduction} \DR_{1,n}(a_1,\ldots,a_n,a_{n+1}) \to \Mcal_{1,n}. \end{equation}
Given $(C,p_1,\ldots,p_n) \in \Mcal_{1,n}$ a choice of lift is a choice of $p_{n+1} \in C \setminus \{ p_1,\ldots,p_n \}$ such that:
\[ \OO_C(a_{n+1} p_{n+1}) \cong \OO_C(-\Sigma_{i=1}^n a_i p_i).\]
Since we work in genus one, the simple expectation is that there are precisely $a_{n+1}^2$ such lifts. However there is a complication: we must exclude the possibility that $p_{n+1} = p_i$ are valid lifts. This amounts to removing the following double ramification loci from $\Mcal_{1,n}$:
\begin{equation} \label{eqn: DR locus in recursion introduction} \DR_{1,n}(a_1,\ldots,a_i+a_{n+1},\ldots,a_n) \subseteq \Mcal_{1,n}. \end{equation}
The proof now proceeds by cut-and-paste. The double ramification loci \eqref{eqn: DR locus in recursion introduction} define a stratification of $\Mcal_{1,n}$ which pulls back to a stratification of $\DR_{1,n}(a_1,\ldots,a_n,a_{n+1})$. We then study the map \eqref{eqn: forgetful map introduction} stratum by stratum; on each locally-closed stratum it is \'etale of calculable degree.

Higher-rank double ramification loci inevitably enter into this argument, since they arise as deeper strata. However, these higher-rank loci do not appear in the final statement of the recursion: after assembling the contributions we observe a remarkable collection of terms, collapsing the formula and producing a purely rank-one statement:

\begin{customthm}{Z}[\Cref{thm: recursion in Grothendieck ring}] \label{thm: recursion introduction} The orbifold Euler characteristic of $\DR_{1,n+1}(a)$ satisfies the following recurrence:
\[ \chi_\orb(\DR_{1,n+1}(a)) = a_{n+1}^2 \, \chi_\orb(\Mcal_{1,n}) - \sum_{i=1}^n \chi_\orb( \DR_{1,n}(a_1,\ldots,a_i + a_{n+1},\ldots,a_n)).\]
\end{customthm}
\noindent The proof of \Cref{thm: rank one introduction} then proceeds by induction on $n$, and is straightforward once the correct formula is guessed.

The higher-rank recursion (\Cref{thm: recursion higher rank}) is no more complicated, however it only applies to matrices $A$ of a special form. We reduce to such matrices using $\GL_r(\Z)$-invariance of the final formula (\Cref{lem: RHS invariant} and \Cref{prop: sufficient to prove for special matrices}). The proof then proceeds by induction on $(r,n)$ in lexicographic order. Again, the difficult step is guessing the correct formula.

Our proof in fact establishes a recurrence in the \'etale Grothendieck ring of orbifolds, the quotient of the Grothendieck ring of orbifolds by the relations
\[ [Y] = [X] \cdot [F] \]
for any \'etale morphism $Y \to X$ with fibre $F$. However, the \'etale Grothendieck ring of orbifolds is isomorphic to $\Q$, the isomorphism being given by the orbifold Euler characteristic: see \cite{ShinderMO} for a proximate argument. This refinement thus contains no additional information.

\begin{remark} The orbifold Euler characteristic, also known as the Euler--Satake characteristic, was introduced by Satake in his foundational work on orbifolds \cite[Section~3]{Satake}. It is additive with respect to stratifications and multiplicative with respect to \'etale covers. It is closely related to the Euler characteristic of classifying spaces \cite{Wall}, a perspective adopted by Harer--Zagier \cite{HarerZagier}. 

The above notion is not to be confused with the \textbf{stringy orbifold Euler characteristic} (confusingly, sometimes also called the orbifold Euler characteristic) arising from physics \cite{DHVW,HirzebruchHoefer}. This is not multiplicative with respect to \'etale covers, and typically differs from both the Euler--Satake characteristic and the naive Euler characteristic.
\end{remark}

\subsection{Intersection theory} Given a normal crossings compactification $\DR_{g,n}^r(A) \subseteq \ol{\DR}_{g,n}^{\, r}(A)$ the logarithmic Gauss--Bonnet--Poincar\'e--Hopf formula identifies
\begin{equation} \label{eqn: integral of log tangent} \chi_\orb(\DR_{g,n}^r(A)) = \int_{\ol{\DR}_{g,n}^{\, r}(A)} c_{\text{top}}( \Theta ) \end{equation}
where $\Theta$ is the logarithmic tangent bundle of the compactification (see e.g. \cite[Proposition~2.1]{CMZ}).

Unfortunately, no such normal crossings compactification has yet been constructed. While compactifications of (higher-rank) double ramification loci have been studied extensively, and play a central role in the (logarithmic) intersection theory of the moduli space of curves and enumerative geometry, these compactifications either contain spurious components or are highly singular at the boundary \cite{Li1, Li2, GathmannThesis, GraberVakil, FaberPandharipandeRelative, MaulikPandharipande, Hain, GZ, BSSZ, JPPZ, HKP, RangProduct, HPS, MarcusWise, PRvZ, JPPZ2, TsengYouHigherGenus, AbreuPacini, MolchoRanganathan, HolmesAJ, HolmesSchmitt, TaleTwo, MolchoSmooth, AbreuPagani, BHPSS, CarocciNabijou1, CH, SpelierPoly,UKR,KannanSong1,KannanSong2, LogDR}.

Focusing on genus one, the triviality of the canonical bundle identifies the double ramification locus with the space of $k$-differentials for any $k$. In rank one, a normal crossings compactification is then given by the space of multiscale differentials \cite{BCGGMMultiscale}. This can be interpreted \cite{BattistellaBozlee,ChenChen} in terms of logarithmic structures and Gorenstein curves, based on the ideas of \cite{RanganathanSantosParkerWise1,RanganathanSantosParkerWise2}. These same ideas can  then be used to produce a normal crossings compactification in higher rank, by combining\cite[Theorem~B]{RanganathanSantosParkerWise2} with rigidification arguments \cite{MaulikPandharipande,UKR}.

Once such a general compactification is constructed, its fundamental class will push forward to a class on a logarithmic blowup of the moduli space of curves. This will differ from the logarithmic double ramification cycle \cite{MolchoRanganathan,LogDR} by boundary corrections arising from the spurious components of the latter.

Reversing the logic, our Theorems~\ref{thm: rank one introduction}~and~\ref{thm: higher rank introduction} calculate the specific class \eqref{eqn: integral of log tangent} of tautological integrals. Recent work of Toh \cite{Toh} studies tautological integrals on the full logarithmic double ramification cycle in rank two, obtaining formulae which also involve greatest common divisors and matrix minors, but not matrix contractions.

This intersection-theoretic strategy has been applied in other contexts. In \cite{GLN} it is used to give a new proof of the Harer--Zagier formula for the orbifold Euler characteristic of the moduli space of curves. Here the description
\[ \Theta = {R}^1 \pi_\star \omega_C^\vee (-\Sigma_{i=1}^n p_i)\]
identifies the orbifold Euler characteristic with an integral on $\overline{\Mcal}_{g,n}$ involving lambda and psi classes. The same strategy also applies to strata of differentials with a single marking \cite[Theorem~1.5]{CSS}. For a calculation strategy closer to our own, involving universal compactified Jacobians, see \cite{Wood}.

\subsection{Higher genus: the Hurwitz stratification} We describe an in-principle method for computing the orbifold Euler characteristic of the double ramification locus in all genus. This method is significantly less efficient than the genus-one recursion employed above, and we are unable to use it to obtain a closed formula. Moreover it does not generalise to higher rank.

The locus $\DR_{g,n}(a)$ is stratified by Hurwitz spaces, which fix the entire ramification profile. A Hurwitz space with $m$ branch points is an \'etale cover of $\Mcal_{0,m}$ of degree equal to the associated Hurwitz number. Using $\chi(\Mcal_{0,m}) = (-1)^{m-3} (m-3)!$ and accounting for the labelling of the branch points, this expresses each $\chi_\orb(\DR_{g,n}(a))$ as a weighted sum of Hurwitz numbers.

In the upcoming \cite{CMS} this method is computer implemented using the packages \texttt{admcycles} and \texttt{diffstrata} \cite{adm, diff} where it is in particular used to experimentally verify \Cref{thm: rank one introduction}.

\subsection*{Notation} For an integer $n \geqslant 1$ we write $[n] \colonequals \{1,\ldots,n\}$.

\subsection*{Acknowledgements} This project began at the University of Bologna, which we thank for excellent working conditions. Across the years we have benefited from discussions on double ramification loci with Francesca~Carocci, Robert~Crumplin, David~Holmes, Sam Molcho, and Dhruv~Ranganathan. We thank Johannes~Schmitt for interesting conversations on the upcoming  related work \cite{CMS}; Terry~Dekun~Song and Evgeny~Shinder for useful clarifications on the (\'etale) Grothendieck ring; and the anonymous referee for valuable expositional comments.

\subsection*{Funding} L.B. is partially supported by the European Union -- NextGenerationEU under the National Recovery and Resilience Plan (PNRR) -- Mission 4: Education and Research -- Component 2: From Research to Business -- Investment 1.1 Notice PRIN 2022 -- DD N. 104 del 2/2/2022, from title "Symplectic varieties: their interplay with Fano manifolds and derived categories", proposal code 2022PEKYBJ – CUP J53D23003840006. L.B. is a member of INdAM group GNSAGA.

\section{Rank one} \label{sec: rank 1} \noindent Fix $g=1$, $n \geqslant 1$ and a ramification vector
\[ a = (a_1,\ldots,a_n) \]
with $\Sigma_{i=1}^n a_i=0$. Let $\Jac_{1,n} \to \Mcal_{1,n}$ denote the universal Jacobian and consider the sections:
\begin{alignat*}{3} 0 \colon & \Mcal_{1,n} \to \Jac_{1,n} \qquad \qquad \qquad \aj_a \colon && \Mcal_{1,n} \to \Jac_{1,n} \\
& (C,p_1,\ldots,p_n) \mapsto \OO_C,  && (C,p_1,\ldots,p_n) \mapsto \OO_C(\Sigma_{i=1}^n a_i p_i).
\end{alignat*}
The double ramification locus $\DR_{1,n}(a)$ is the fibre product:
\[
\begin{tikzcd}
\DR_{1,n}(a) \ar[r] \ar[d] \ar[rd,phantom,"\square"] & \Mcal_{1,n} \ar[d,"\aj_a"] \\
\Mcal_{1,n} \ar[r,"0"] & \Jac_{1,n}.	
\end{tikzcd}
\]
By \cite[Proposition~1.2]{SchmittDimension} this is a smooth hypersurface in $\Mcal_{1,n}$ (though possibly empty, and excluding the trivial case $a=(0,\ldots,0)$). At the level of closed points:
\[ \DR_{1,n}(a) = \left\{ (C,p_1,\ldots,p_n) \in \mathcal{M}_{1,n} \mid \OO_C(\Sigma_{i=1}^n a_i p_i) \cong \OO_C \right\}. \]
The main result of this section is:

\begin{theorem}[\Cref{thm: rank one introduction}] \label{thm: Euler char rank one} The orbifold Euler characteristic of $\DR_{1,n}(a)$ is given by:
\begin{equation} \label{eqn: rank 1 formula} \chi_\orb(\DR_{1,n}(a)) = \dfrac{(-1)^{n-1}(n-1)!}{24} \left( \sum_{i=1}^n a_i^2 - 2 \right). \end{equation}	
\end{theorem}

This result will be deduced from the following recurrence relation.
\begin{theorem}[\Cref{thm: recursion introduction}] \label{thm: recursion in Grothendieck ring} Fix $n \geqslant 1$ and a length $n\! +\! 1$ ramification vector $a=(a_1,\ldots,a_n,a_{n+1})$. For each $i \in [n]$ define the following length $n$ ramification vector:
\[ a(i) \colonequals (a_1,\ldots,a_{i-1},a_i+a_{n+1},a_{i+1},\ldots,a_n).\]
Then the orbifold Euler characteristic of $\DR_{1,n+1}(a)$ satisfies the following recurrence relation:
\begin{equation} \label{eqn: recurrence relation Grothendieck ring} \chi_\orb(\DR_{1,n+1}(a)) = a_{n+1}^2 \, \chi_\orb(\Mcal_{1,n}) - \sum_{i=1}^n \chi_\orb(\DR_{1,n}(a(i))).\end{equation}
\end{theorem}

\begin{proof}
For $a_{n+1}=0$ the result follows immediately by studying the fibres of the smooth and representable morphism:
\[ \DR_{1,n+1}(a_1,\ldots,a_n,0) \to \DR_{1,n}(a_1,\ldots,a_n). \]
Thus assume $a_{n+1} \neq 0$ and consider the representable morphism forgetting the final marking:
\begin{equation} \label{eqn: proof rank 1 forgetful map 2} \DR_{1,n+1}(a_1,\ldots,a_n,a_{n+1}) \to \Mcal_{1,n}.
\end{equation}
We begin with the $n=1$ case, which is instructive. We claim that the morphism \eqref{eqn: proof rank 1 forgetful map 2} is \'etale of degree $a_2^2-1$. Indeed we have $a_2=-a_1$, and given $(C,p_1) \in \Mcal_{1,1}$ a lift consists of a choice of $p_2 \in C$ such that $\OO_C(a_2 p_2) \cong \OO_C(a_2 p_1)$ and $p_2 \neq p_1$. There are precisely $a_2^2-1$ of these, and we obtain the relation
\begin{equation} \label{eqn: proof Grothendieck relation n equals 1 case} \chi_\orb(\DR_{1,2}(a_1,a_2)) = (a_2^2-1) \chi_\orb(\Mcal_{1,1}).\end{equation}
In this case we only have $a(1)=(a_1+a_2)=(0)$ and so $\DR_{1,1}(a(1)) = \Mcal_{1,1}$. Therefore \eqref{eqn: proof Grothendieck relation n equals 1 case} is equivalent to \eqref{eqn: recurrence relation Grothendieck ring} and this completes the $n=1$ case.

For $n \geqslant 2$, however, the morphism \eqref{eqn: proof rank 1 forgetful map 2} is not \'etale. We will now define stratifications of the source and target, and show that the morphism is \'etale on each locally-closed stratum, with degree depending on the stratum.

We first define the stratification on $\Mcal_{1,n}$. Recall for $i \in [n]$ the length $n$ ramification vector:
\[ a(i) \colonequals (a_1,\ldots,a_{i-1},a_i+a_{n+1},a_{i+1},\ldots,a_n).\]
We define the depth-$1$ closed strata in $\Mcal_{1,n}$ to be:
\[ \DR_{1,n}(a(1)),\ldots,\DR_{1,n}(a(n)).\]
More generally, the depth-$k$ closed strata are indexed by subsets $I = \{ i_1,\ldots,i_k\} \subseteq [n]$ of size $k$ and given by:
\[ \DR_{1,n}(a(I)) = \DR_{1,n} \! \begin{bmatrix} a(i_1) \\ \vdots \\ a(i_k) \end{bmatrix} \colonequals \bigcap_{j=1}^k \DR_{1,n}(a(i_j)).\]
Note that these intersections are often not dimensionally transverse, for instance if two rows admit a common divisor. Therefore a depth-$k$ stratum may have smaller codimension than $k$. 

We put a partial order on the closed strata $\DR_{1,n}(a(I))$ using the inclusion order on the index sets $I$ (note that inside $\Mcal_{1,n}$ there may be additional containments of closed strata beyond those forced by the index sets). From the closed strata we obtain locally-closed strata by removing the deeper strata:
\[ \DR_{1,n}^\circ(a(I)) \colonequals \DR_{1,n}(a(I)) \setminus \bigcup_{I \subsetneq J \subseteq [n]} \DR_{1,n}(a(J)). \]
Some locally-closed strata may be empty, but this does not affect the argument.

This stratification of $\Mcal_{1,n}$ pulls back along \eqref{eqn: proof rank 1 forgetful map 2} to give a stratification of $\DR_{1,n+1}(a)$ which we describe explicitly. The depth-$1$ closed strata in this stratification are
\[ \DR_{1,n+1} \! \left[ \begin{array}{cccc} \!\!\! a_1 & \cdots & a_n & a_{n+1} \!\!\!\! \\ \multicolumn{3}{c}{\!\!\! \mbox{-----}a(i)\mbox{-----}} & 0 \!\!\!\! \end{array} \right] \]
while the depth-$k$ closed strata consist of the intersections. We use the same notation as before for the locally-closed strata. Restricting \eqref{eqn: proof rank 1 forgetful map 2} to a locally-closed stratum in $\DR_{1,n+1}(a)$ of depth $k$ we obtain a map:
\[ \DR_{1,n+1}^\circ \! \left[ \begin{array}{cccc} \!\!\! a_1 & \cdots & a_n & a_{n+1} \!\!\!\! \\ \multicolumn{3}{c}{\!\!\! \mbox{-----}a(i_1)\mbox{-----}} & 0 \!\!\!\! \\ \!\!\! & \vdots & & \vdots \!\!\!\! \\ \multicolumn{3}{c}{\!\!\! \mbox{-----}a(i_k)\mbox{-----}} & 0 \!\!\!\! \end{array} \right] \to \DR_{1,n}^\circ \! \begin{bmatrix} \mbox{-----}a(i_1)\mbox{-----} \\ \vdots \\ \mbox{-----}a(i_k)\mbox{-----} \end{bmatrix}.
 \]
We claim that this map is \'etale of degree $a_{n+1}^2-k$. Fix a point $(C,p_1,\ldots,p_n)$ in the target. Then a lift consists of a choice of point $p_{n+1} \in C$ such that
\begin{equation} \label{eqn: condition for lifting p} \OO_C(a_{n+1} p_{n+1}) \cong \OO_C(-\Sigma_{i=1}^n a_i p_i).\end{equation}
There are $a_{n+1}^2$ of these, but we need to examine the possibilities $p_{n+1}=p_i$. Let $I=\{i_1,\ldots,i_k\} \subseteq [n]$ be the set indexing the given stratum. For $i \in I$ the point $(C,p_1,\ldots,p_n)$ in the target satisfies
\[ \OO_C(a_1 p_1 + \ldots + a_{i-1}p_{i-1} + (a_i+a_{n+1})p_i + a_{i+1}p_{i+1} + \ldots + a_n p_n) \cong \OO_C \]
because $a(i)$ appears as a row in the target matrix. It follows that the choice $p_{n+1}=p_i$ satisfies the ramification condition \eqref{eqn: condition for lifting p}. On the other hand for $i \in [n] \setminus I$ the fact that we have removed the intersections with deeper strata ensures that
\[ \OO_C(a_1 p_1 + \ldots + a_{i-1}p_{i-1} + (a_i+a_{n+1})p_i + a_{i+1}p_{i+1} + \ldots + a_n p_n) \not\cong \OO_C \]
and therefore the choice $p_{n+1}=p_i$ does not satisfy the ramification condition \eqref{eqn: condition for lifting p}. It follows that we must remove $|I|=k$ of the choices of $p_{n+1}$ and that the remaining choices give valid lifts. This shows that the map is \'etale of degree $a_{n+1}^2-k$ as claimed. We obtain an identity:
\[ \chi_\orb (\DR_{1,n+1}^\circ \! \left[ \begin{array}{cccc} \!\!\! a_1 & \cdots & a_n & a_{n+1} \!\!\!\! \\ \multicolumn{3}{c}{\!\!\!\mbox{-----}a(i_1)\mbox{-----}} & 0 \!\!\!\! \\ \!\!\! & \vdots & & \vdots \!\!\!\! \\ \multicolumn{3}{c}{\!\!\!\mbox{-----}a(i_k)\mbox{-----}} & 0 \!\!\!\! \end{array} \right]) = (a_{n+1}^2-k) \, \chi_\orb (\DR_{1,n}^\circ \! \begin{bmatrix} \mbox{-----}a(i_1)\mbox{-----} \\ \vdots \\ \mbox{-----}a(i_k)\mbox{-----} \end{bmatrix}).
 \]
The above formula is also valid when $a_{n+1}^2-k < 0$, for in this case we claim that both source and target strata are empty. Indeed, suppose otherwise and choose a point $(C,p_1,\ldots,p_n)$ in the target stratum. The ramification conditions imply that
\[ \OO_C(a_{n+1}p_i) \cong \OO_C(a_{n+1}p_j) \]
for all $i,j \in I$. There are at most $a_{n+1}^2$ such points of $C$, and it follows that $k \leqslant a_{n+1}^2$ which contradicts the assumption. We conclude that if $a_{n+1}^2 - k < 0$ then the source and target strata are empty, so the above formula trivially holds.

We now employ the scissor relations with respect to the stratification of the source of \eqref{eqn: proof rank 1 forgetful map 2}:
\begin{align} \label{eqn: proof rank 1 intermediate Grothendieck ring relation} \nonumber \chi_\orb( \DR_{1,n+1}(a) ) & = \sum_{k=0}^n \sum_{\{i_1,\ldots,i_k\} \subseteq [n]} \chi_\orb( \DR_{1,n+1}^\circ \! \left[ \begin{array}{cccc} \!\!\! a_1 & \cdots & a_n & a_{n+1} \!\!\!\! \\ \multicolumn{3}{c}{\!\!\! \mbox{-----}a(i_1)\mbox{-----}} & 0 \!\!\!\! \\ & \!\!\! \vdots & & \vdots \!\!\!\! \\ \multicolumn{3}{c}{\!\!\! \mbox{-----}a(i_k)\mbox{-----}} & 0 \!\!\!\! \end{array} \right]) \\[0.3cm]
& = \sum_{k=0}^n \sum_{\{i_1,\ldots,i_k\} \subseteq [n]} (a_{n+1}^2-k) \, \chi_\orb( \DR_{1,n}^\circ \! \begin{bmatrix} \mbox{-----}a(i_1)\mbox{-----} \\ \vdots \\ \mbox{-----}a(i_k)\mbox{-----} \end{bmatrix}). \end{align}
Having used the scissor relations to deconstruct the source, we now use them to reconstruct the target. The key relations are:
\begin{align*} \chi_\orb( \Mcal_{1,n} ) & = \sum_{k=0}^n \sum_{\{i_1,\ldots,i_k\} \subseteq [n]} \chi_\orb( \DR_{1,n}^\circ \! \begin{bmatrix} \mbox{-----}a(i_1)\mbox{-----} \\ \vdots \\ \mbox{-----}a(i_k)\mbox{-----} \end{bmatrix}), \\[0.3cm]
\chi_\orb( \DR_{1,n}(a(i)) ) & = \sum_{k=1}^n \sum_{\substack{\{i_1,\ldots,i_k\} \subseteq [n] \\ i \in \{i_1,\ldots,i_k\}}} \chi_\orb( \DR_{1,n}^\circ \! \begin{bmatrix} \mbox{-----}a(i_1)\mbox{-----} \\ \vdots \\ \mbox{-----}a(i_k)\mbox{-----} \end{bmatrix}).
\end{align*}
Examining \eqref{eqn: proof rank 1 intermediate Grothendieck ring relation} we see that each term of the form
\[ k \cdot \chi_\orb( \DR_{1,n}^\circ \! \begin{bmatrix} \mbox{-----}a(i_1)\mbox{-----} \\ \vdots \\ \mbox{-----}a(i_k)\mbox{-----} \end{bmatrix}) \]
participates in precisely $k$ of the $\chi_\orb(\DR_{1,n}(a(i)))$. Assembling, we conclude from \eqref{eqn: proof rank 1 intermediate Grothendieck ring relation} that
\[ \chi_\orb( \DR_{1,n+1}(a) ) = a_{n+1}^2 \, \chi_\orb( \Mcal_{1,n} ) - \sum_{i=1}^n \chi_\orb( \DR_{1,n}(a(i)) ) \]
as required.
\end{proof}

\begin{proof}[Proof of \Cref{thm: Euler char rank one}] We induct on $n$. The base case $n=1$ is trivial; we must have $a_1=0$ and then
\[ \chi_\orb(\DR_{1,1}(0)) = \chi_\orb(\Mcal_{1,1}) = -1/12 \]
by the Harer--Zagier formula (\Cref{lem: HarerZagier}), which agrees with \eqref{eqn: rank 1 formula}. Now assume the claim holds for ramification vectors of length $n$ and consider a ramification vector $a=(a_1,\ldots,a_n,a_{n+1})$ of length $n\! + \! 1$. We then have
\begin{align*} \chi_\orb( \DR_{1,n+1}(a) ) & = a_{n+1}^2 \, \chi_\orb( \Mcal_{1,n} ) - \sum_{i=1}^n \chi_\orb( \DR_{1,n}(a(i)) ) \\
& = a_{n+1}^2 \left( \dfrac{(-1)^n (n-1)!}{12} \right) - \sum_{i=1}^n \dfrac{(-1)^{n-1}(n-1)!}{24} \left( \Sigma_{j \neq i} a_j^2 + (a_i + a_{n+1})^2 - 2 \right) \\[0.2cm]
& = \dfrac{(-1)^n (n-1)!}{24} \left( 2a_{n+1}^2 + \sum _{i=1}^n \left( \Sigma_{j=1}^{n+1} a_j^2 + 2a_i a_{n+1} - 2 \right) \right) \\[0.2cm]
& = \dfrac{(-1)^n (n-1)!}{24} \left( n \left( \Sigma_{j=1}^{n+1} a_j^2 - 2 \right) + 2a_{n+1}^2 + 2a_{n+1} \Sigma_{i=1}^n a_i  \right) \\[0.2cm]
& = \dfrac{(-1)^n n!}{24} \left( \sum_{i=1}^{n+1} a_i^2 - 2 \right)
\end{align*}
where the first equality follows from Theorem~\ref{thm: recursion in Grothendieck ring}, the second equality follows from the Harer--Zagier formula (\Cref{lem: HarerZagier}), and the induction hypothesis, and the last equality follows from the fact that $\Sigma_{i=1}^{n+1} a_i = 0$. This completes the induction step.
\end{proof}

The above proof uses the Harer--Zagier formula \cite{HarerZagier} for the orbifold Euler characteristic of the moduli space of curves. In genus one this admits an elementary proof, presumably well-known to experts, which we include for completeness.

\begin{lemma}[Harer--Zagier in genus one] \label{lem: HarerZagier} For $n \geqslant 1$ we have:
\[ \chi_\orb(\Mcal_{1,n}) = \dfrac{(-1)^n (n-1)!}{12}.\]
\end{lemma}

\begin{proof} We proceed by induction on $n$. For the base case we note that $\Mcal_{1,1}$ has $\Aaff^{\!1}$ as its coarse moduli space, with the general point having an automorphism group of order $2$, and two special points $\xi_{1728}$ and $\xi_0$ having automorphism groups of orders $4$ and $6$. We thus have:
\[ \chi_\orb(\Mcal_{1,1}) = \frac{1}{2}  \, \chi(\Aaff^{\!1} \setminus \{\xi_{1728},\xi_0\}) + \frac{1}{4} \, \chi(\xi_{1728}) + \frac{1}{6} \, \chi(\xi_0) = -\frac{1}{2} + \frac{1}{4} + \frac{1}{6} = -\frac{1}{12}.\]
For the induction step, consider the forgetful morphism $\Mcal_{1,n+1} \to \Mcal_{1,n}$. This morphism is representable, and each fibre is a genus-one curve $C$ with the points $p_1,\ldots,p_n$ removed. We conclude:
\begin{align*} \chi_\orb(\Mcal_{1,n+1}) & = \chi(C \setminus \{p_1,\ldots,p_n\}) \cdot \chi_\orb(\Mcal_{1,n}) \\
& = (-n) \cdot \dfrac{(-1)^n (n-1)!}{12} \\
& = \dfrac{(-1)^{n+1} ((n+1)-1)!}{12}. \qedhere
\end{align*}
\end{proof}

\section{Higher rank} \label{sec: higher rank}

\noindent We proceed to the higher-rank case. The recursion generalises directly (\Cref{thm: recursion higher rank}) and the induction strategy still applies. The difficulty is guessing the correct formula.

Fix $g=1$, $r \geqslant 1$, $n \geqslant 1$ and an $r \times n$ integer matrix
\[ A = 
\begin{bmatrix}
a^{(1)}_1 & \cdots & a^{(1)}_n \\
 & \vdots &  \\
a^{(r)}_1 & \cdots & a^{(r)}_n	
\end{bmatrix}
\]
such that each row sums to zero: $\Sigma_{i=1}^n a^{(j)}_i = 0$ for all $j \in [r]$. We refer to this as a \textbf{double ramification matrix}. Given an $r \times n$ double ramification matrix $A$, the associated \textbf{rank-$r$ double ramification locus} is denoted and defined: 
\[ \DR^r_{1,n}(A) \colonequals \bigcap_{i=1}^r \DR_{1,n}(a^{(i)}) \subseteq \Mcal_{1,n} \]
where $a^{(i)}$ denotes the $i$th row of $A$. Its closed points correspond to marked curves $(C,p_1,\ldots,p_n)$ satisfying the $r$ simultaneous equations:
\begin{align*}
\OO_C(\Sigma_{i=1}^n a^{(1)}_i & p_i) \cong \OO_C, \\
\vdots \\
\OO_C(\Sigma_{i=1}^n a^{(r)}_i & p_i) \cong \OO_C.
\end{align*}
The main result of this section (\Cref{thm: higher rank DR}) gives a formula for the orbifold Euler characteristic of this locus.

\begin{remark}
While $\DR_{1,n}^r(A) \subseteq \Mcal_{1,n}$ has expected dimension $n-r$ its actual dimension may be larger, for instance if some rows of $A$ are linearly dependent over $\Z$. It can also be empty, for instance if $A$ contains a row of the form $(1,-1,0,\ldots,0)$. The formula below for the orbifold Euler characteristic holds in all cases. 
\end{remark}

In \Cref{sec: higher rank formula} we state the formula (\Cref{thm: higher rank DR}). In \Cref{sec: reduction via GL} we use $\GL_r(\Z)$-invariance to reduce to a special class of double ramification matrices (\Cref{prop: sufficient to prove for special matrices}), and in \Cref{sec: higher rank recursion} we establish an orbifold Euler characteristic recursion for these matrices (\Cref{thm: recursion higher rank}). In \Cref{sec: matrix lemmas} we establish an important lemma on the linear algebra of double ramification matrices (\Cref{lem: matrix lemma 2}). This is used in the proof of the formula, which is given in \Cref{sec: higher rank proof}. Having obtained the formula, in \Cref{sec: leading term} we provide a simplification of its leading term, and finally in \Cref{sec: comparing theorems X and Y} we compare it to the rank-one formula obtained in the previous section.

\subsection{Formula} \label{sec: higher rank formula} We establish the necessary notation. A \textbf{partition} $\Ical \vdash [n]$ is an unordered collection of subsets
\[ \Ical = \{ I_1,\ldots,I_{\ell(\Ical)} \} \]
with each $I_j \neq \emptyset$ and $[n]=I_1 \sqcup \cdots \sqcup I_{\ell(\Ical)}$. Given a partition $\Ical \vdash [n]$ the associated \textbf{contraction} of $A$ is obtained by summing the columns associated to each part of $\Ical$,
\[ A_\Ical \colonequals 
\begin{bmatrix}
a^{(1)}_{I_1} & \cdots & a^{(1)}_{I_{\ell(\Ical)}} \\
 & \vdots &  \\
a^{(r)}_{I_1} & \cdots & a^{(r)}_{I_{\ell(\Ical)}}
\end{bmatrix}
\]
where $a_I^{(j)} \colonequals \Sigma_{i \in I} a^{(j)}_i$ for any subset $I \subseteq [n]$. The contraction $A_\Ical$ is well-defined up to permutation of the columns and is an $r \times \ell(\Ical)$ double ramification matrix. For $0 \leqslant k \leqslant \min (r, \ell(\Ical) )$ we then define
\[ G_{k \times k}(A_\Ical) \in \Z \]
to be the greatest common divisor of all the $k \times k$ minors of $A_{\Ical}$. This is well-defined up to sign. By convention we take:
\[ G_{0 \times 0}(A_\Ical)=1, \qquad \gcd(m_1,\ldots,m_l,0) = \gcd(m_1,\ldots,m_l), \qquad \gcd(\emptyset)=0.\]
We are now ready to state the main result.

\begin{theorem}[\Cref{thm: higher rank introduction}] \label{thm: higher rank DR} Fix an $r \times n$ double ramification matrix $A$. The orbifold Euler characteristic of the associated rank $r$ double ramification locus is:
\begin{equation} \label{eqn: higher rank formula} \chi_\orb(\DR^r_{1,n}(A)) = \dfrac{(-1)^n}{12} \sum_{k=0}^r \sum_{\substack{\Ical \, \vdash [n] \\ \ell(\Ical)=k+1}} (-1)^k (\# I_1-1)! \cdots (\# I_{k+1}-1)! \cdot G_{k \times k}(A_\Ical)^2.\end{equation}
\end{theorem}

\begin{remark} If $n \leqslant r$, then there are no partitions $\Ical \vdash [n]$ of length $k+1$ for $n \leqslant k \leqslant r$. The associated terms in the above formula simply vanish.
\end{remark}

\subsection{Reduction via $\GL_r(\Z)$-invariance} \label{sec: reduction via GL} Given an $r \times n$ double ramification matrix $A$ and a matrix $M \in \GL_r(\Z)$, the product $MA$ is again an $r \times n$ double ramification matrix, since elementary row operations preserve this property. Clearly we have
\begin{equation} \label{eqn: GL invariance of DR locus} \DR_{1,n}^r(A) = \DR_{1,n}^r(MA) \end{equation}
as substacks of $\Mcal_{1,n}$. We will use this $\GL_r(\Z)$-invariance to reduce to a special class of double ramification matrices. The key fact is the following:

\begin{lemma} \label{lem: RHS invariant} The right-hand side of \eqref{eqn: higher rank formula} is $\GL_r(\Z)$-invariant.
\end{lemma}

\begin{proof} Taking contractions commutes with the action of $\GL_r(\Z)$, that is
\[ (MA)_\Ical = M (A_\Ical) \]
for every $\Ical \vdash [n]$ and $M \in \GL_r(\Z)$. Each $k \times k$ minor of $M(A_\Ical)$ is a $\Z$-linear combination of $k \times k$ minors of $A_\Ical$. Therefore
\[ G_{k \times k}(A_\Ical) \mid G_{k \times k}(M (A_\Ical)).\]
But the same argument applied to $M^{-1}$ shows the reverse divisibility, so  in fact 
\[ G_{k \times k}(A_\Ical) = G_{k \times k}(M (A_\Ical)). \qedhere \]
\end{proof}

Consequently, to prove \Cref{thm: higher rank DR} for a double ramification matrix $A$ it is sufficient to prove it for $MA$ for a single $M \in \GL_r(\Z)$. We use the following reduction.

\begin{lemma} \label{lem: normal form for matrix} Given an $r \times n$ integer matrix $A$, there exists $M \in \GL_r(\Z)$ such that $MA$ takes the following special form:
\begin{equation} \label{eqn: special form for DR matrix}
\begingroup
\renewcommand*{\arraystretch}{1.3}
\begin{bmatrix}
a^{(1)}_1 & \cdots & a^{(1)}_{n-1} & a^{(1)}_n \\
a^{(2)}_1 & \cdots & a^{(2)}_{n-1} & 0 \\
& \vdots & & \vdots \\
a^{(r)}_1 & \cdots & a^{(r)}_{n-1} & 0
\end{bmatrix}.
\endgroup
\end{equation}
\end{lemma}

\begin{proof}
We prove the result for $r=2$; the general case proceeds by induction on the rows. We have:
\[ A = \begin{bmatrix} a_1 & \cdots & a_{n-1} & a_n \\ b_1 & \cdots & b_{n-1} & b_n \end{bmatrix}. \]
Take $d \colonequals \gcd(a_n,b_n)$ with an arbitrary choice of sign. There exist $p,q \in \Z$ with
\[ p a_n + q b_n = d. \]
Dividing through by $d$, we obtain $s,t \in \Z$ with
\begin{equation} \label{eqn: ps and qt} ps + qt = 1.\end{equation}
Consider the matrix
\[ M^\prime \colonequals \begin{bmatrix} p & q \\ -t & s \end{bmatrix} \]
which belongs to $\GL_2(\Z)$ by \eqref{eqn: ps and qt}. The matrix $M^\prime A$ takes the following form:
\[ \begin{bmatrix} \mbox{-----}\star\mbox{-----} & p a_n + q b_n \\ \mbox{-----}\star\mbox{-----} & -t a_n + s b_n \end{bmatrix} = \begin{bmatrix} \mbox{-----}\star\mbox{-----} & d \\ \mbox{-----}\star\mbox{-----} & -t a_n + s b_n \end{bmatrix}.\]
Since $d \mid a_n$ and $d \mid b_n$ we have $d \mid -t a_n + s b_n$.  Thus we may add an appropriate multiple of the first row to the second row to obtain a matrix of the form
\[ \begin{bmatrix} \mbox{-----}\star\mbox{-----} & d \\ \mbox{-----}\star\mbox{-----} & 0 \end{bmatrix}\]
as required.
\end{proof}

\begin{proposition} \label{prop: sufficient to prove for special matrices} To prove \Cref{thm: higher rank DR}, it is sufficient to prove it for matrices of the form \eqref{eqn: special form for DR matrix}.
\end{proposition}

\begin{proof} Combine \eqref{eqn: GL invariance of DR locus} with Lemmas~\ref{lem: RHS invariant} and \ref{lem: normal form for matrix}.\end{proof}

\subsection{Geometric recursion} \label{sec: higher rank recursion} Having reduced to matrices of the form \eqref{eqn: special form for DR matrix}, we now establish the following recursion generalising \Cref{thm: recursion in Grothendieck ring}:

\begin{theorem} \label{thm: recursion higher rank} Fix $n \geqslant 1$ and consider an $r \times (n+1)$ double ramification matrix $A$ of the special form \eqref{eqn: special form for DR matrix}, writing $a=(a_1,\ldots,a_n,a_{n+1})$ for the first row and $B$ for the $(r-1) \times n$ submatrix in the bottom left corner:
\[ A = 
\begin{bmatrix}  
a_1 & \cdots & a_n & a_{n+1} \\
	& & & 0 \\
	& B & & \vdots \\
	& & & 0
 \end{bmatrix}.
\]
For each $i \in [n]$ define the length $n$ ramification vector:
\[ a(i) \colonequals (a_1,\ldots,a_{i-1},a_i+a_{n+1},a_{i+1},\ldots,a_n).\]
Then the orbifold Euler characteristic of $\DR^r_{1,n+1}(A)$ satisfies the following recurrence:
\begin{equation} \label{eqn: recursion higher rank}
\chi_\orb( \DR_{1,n+1}^r(A) ) = a_{n+1}^2 \, \chi_\orb( \DR_{1,n}^{r-1}(B) ) - \sum_{i=1}^n \chi_\orb ( \DR_{1,n}^r \begin{bmatrix} \mbox{-----}a(i)\mbox{-----} \\ B \end{bmatrix} ). \end{equation}
\end{theorem}

\begin{proof} The proof of \Cref{thm: recursion in Grothendieck ring} applies \emph{mutatis mutandis}. The analogue of the forgetful morphism $\DR_{1,n+1}(a) \to \Mcal_{1,n}$ is
\[ \DR_{1,n+1}^r(A) \to \DR_	{1,n}^{r-1}(B) \]
and the construction of the stratification is identical.
\end{proof}

\subsection{Matrix lemmas} \label{sec: matrix lemmas} The proof will proceed by induction, using \Cref{prop: sufficient to prove for special matrices} and \Cref{thm: recursion higher rank}. We require a basic result on the linear algebra of double ramification matrices.

\begin{lemma} \label{lem: matrix lemma 1} Consider an $r \times (r+1)$ double ramification matrix $A$. Then all the $r \times r$ minors of $A$ coincide up to sign.
\end{lemma}

\begin{proof}
Consider such a double ramification matrix:
\[ A = 
\begin{bmatrix}
a^{(1)}_1 & \cdots & a^{(1)}_{r+1} \\
 & \vdots &  \\
a^{(r)}_1 & \cdots & a^{(r)}_{r+1}
\end{bmatrix}.
\]
An $r \times r$ minor is obtained by removing a single column. Given indices $i, j \in [r+1]$ with $i < j$ we consider the contraction $B$ of $A$ obtained by summing the $i$th and $j$th columns:
\[
B \colonequals 
\! \left[
\begin{array}{cccccccc}
\!\!\! \vline & & \vline & & \vline & & \vline & \vline \!\!\!\!  \\
\!\!\! a_1 & \cdots & \widehat{a_i} & \cdots & \widehat{a_j} & \cdots & a_r & a_i\!+\!a_j \!\!\!\! \\
\!\!\! \vline & & \vline & & \vline & & \vline & \vline \!\!\!\! 
\end{array}
\right].
\]
This is an $r \times r$ double ramification matrix, so the sum of the columns is equal to the zero vector and so $\det B=0$. But on the other hand by multilinearity of the determinant $\det B$ is equal to:
\[ \det \! \left[
\begin{array}{cccccccc}
\!\! \vline & & \vline & & \vline & & \vline & \vline \!\!\! \\
\!\! a_1 & \cdots & \widehat{a_i} & \cdots & \widehat{a_j} & \cdots & a_r & a_i \!\!\! \\
\!\! \vline & & \vline & & \vline & & \vline & \vline \!\!\! 
\end{array}
\right]
+
\det \! \left[
\begin{array}{cccccccc}
\!\! \vline & & \vline & & \vline & & \vline & \vline \!\!\!  \\
\!\! a_1 & \cdots & \widehat{a_i} & \cdots & \widehat{a_j} & \cdots & a_r & a_j \!\!\! \\
\!\! \vline & & \vline & & \vline & & \vline & \vline \!\!\! 
\end{array}
\right].
\]
Up to signs determined by the appropriate column permutations, these two terms are the $r \times r$ minors of $A$ corresponding to $i$ and $j$.
\end{proof}

\begin{corollary} \label{lem: matrix lemma 2} For $k \in \{0,\ldots,r\}$ consider an $r \times (k+1)$ double ramification matrix $A$, so that the $k \times k$ minors of $A$ are obtained by deleting $r-k$ rows and $1$ column. Then up to sign, the $k \times k$ minor depends only on the choice of rows and not on the choice of column.
\end{corollary}

\begin{proof} This follows immediately from \Cref{lem: matrix lemma 1}: deleting $r-k$ rows produces a $k \times (k+1)$ double ramification matrix.	
\end{proof}

\subsection{Proof} \label{sec: higher rank proof}

\begin{proof}[Proof~of~\Cref{thm: higher rank DR}]
We induct on the pair $(r,n)$ using the lexicographic order. Given $(r,n)$ we assume that the formula has already been established for pairs $(r^\prime,n^\prime)$ such that either:
\begin{enumerate}
\item $r^\prime < r$; or
\item $r^\prime=r$ and $n^\prime<n$.	
\end{enumerate}
The base case is when $r$ is arbitrary and $n=1$. Then $A$ is a column vector consisting of $r$ zeros, so $\DR_{1,1}^r(A) = \Mcal_{1,1}$. In this case by the Harer--Zagier formula:
\[ \chi_\orb(\DR_{1,1}^r(A)) = -1/12.\]
On the other hand in the formula \eqref{eqn: higher rank formula} there is a single partition $\Ical \vdash [1]$ and by convention we have $G_{0 \times 0}(A_\Ical)=1$. The total contribution is $-1/12$, verifying the base case.

For the induction step, consider an $r \times(n+1)$ double ramification matrix $A$. By \Cref{prop: sufficient to prove for special matrices} we may assume $A$ takes the following form
\begin{equation} \label{eqn: matrix A special form}
A=
\begin{bmatrix}  
a_1 & \cdots & a_n & a_{n+1} \\
	& & & 0 \\
	& B & & \vdots \\
	& & & 0
 \end{bmatrix}.
\end{equation}
and then \Cref{thm: recursion higher rank} gives:
\begin{equation} \label{eqn: proof general case recursion of Euler chars}
\chi_\orb(\DR_{1,n+1}^r(A)) = a_{n+1}^2 \, \chi_\orb(\DR_{1,n}^{r-1}(B)) - \sum_{i=1}^n \chi_\orb(\DR_{1,n}^r \! \begin{bmatrix} \mbox{-----}a(i)\mbox{-----} \\ B \end{bmatrix}).
\end{equation}
We apply the induction hypothesis to the right-hand side. The following definition will be useful. A partition $\Ical \vdash [n\!+\!1]$ is \textbf{lonely} if $n\!+\!1$ constitutes an entire part, and \textbf{friendly} otherwise. In the right-hand side above, the first term will provide the contributions of the lonely partitions, while the second term will provide the contributions of the friendly partitions.

We begin with $\chi_\orb(\DR_{1,n}^{r-1}(B))$. The contributions are indexed by partitions $\Ical = \{ I_1,\ldots,I_{k+1} \} \vdash [n]$ of length $k\!+\!1$ for $k=0,\ldots,r-1$. These correspond bijectively with lonely partitions 
\[ \Ical^\prime \colonequals \{ I_1,\ldots,I_{k+1},\{n\!+\!1\} \} \vdash [n\!+\!1] \]
of length $k\!+\!2$. Ranging over all $k$, we obtain all lonely partitions $\Ical \vdash [n\!+\!1]$ of length $k+1$ for $k=1,\ldots,r$ (and since there are no lonely partitions of length $1$, we may in fact say for $k=0,\ldots,r$).

Given such an $\Ical \vdash [n]$ and corresponding lonely partition $\Ical^\prime \vdash [n\!+\!1]$, the associated contractions are related as follows:
\[
A_{\Ical^\prime} = 
\begin{bmatrix} a_{I_1} & \cdots & a_{I_{k+1}} & a_{n+1} \\
& & & 0 \\
& B_\Ical & & \vdots \\
& & & 0 	
 \end{bmatrix}.
\]
We must now compare the $k \times k$ minors of $B_\Ical$ with the $(k+1) \times (k+1)$ minors of $A_{\Ical^\prime}$.

The $(k+1) \times (k+1)$ minors of $A_{\Ical^\prime}$ are obtained by selecting $(k+1)$ rows and $(k+1)$ columns, but up to sign the choice of columns does not matter by \Cref{lem: matrix lemma 2}. If the first row is not among the $(k+1)$ selected rows, then we may include the final column among the $(k+1)$ selected columns: the resulting submatrix has a column of zeros, and hence the minor vanishes.

To obtain a nonzero minor of $A_{\Ical^\prime}$ we must therefore include the first row among the $(k+1)$ selected rows. This amounts to choosing $k$ rows of $B_\Ical$. Once this is done, we can make an arbitrary choice of $(k+1)$ columns of $A_{\Ical^\prime}$ by \Cref{lem: matrix lemma 2}. We choose $k$ columns of $B_\Ical$ together with the final column of $A_{\Ical^\prime}$.

In this way we obtain a bijection between the nonzero $(k+1) \times (k+1)$ minors of $A_{\Ical^\prime}$ and the $k \times k$ minors of $B_\Ical$. Expanding along the final column we see that these are related, up to sign, by the factor $a_{n+1}$. We conclude:
\[ G_{k+1 \times k+1}(A_{\Ical^\prime}) = a_{n+1} G_{k \times k}(B_\Ical). \]
Examining the first term on the right-hand side of \eqref{eqn: proof general case recursion of Euler chars} we obtain precisely the lonely contributions:
\begin{align} \nonumber a_{n+1}^2 \, \chi_\orb(\DR_{1,n}^{r-1}(B)) & = \dfrac{(-1)^n}{12} \sum_{k=0}^{r-1} \sum_{\substack{\Ical \vdash [n] \\ \ell(\Ical)=k+1}} (-1)^k (\# I_1-1)! \cdots (\# I_{k+1}-1)! \cdot a_{n+1}^2 G_{k \times k}(B_\Ical)^2 \\[0.2cm]
\nonumber & = \dfrac{(-1)^n}{12} \sum_{k=0}^r \sum_{\substack{\Ical \vdash [n + 1] \\ \ell(\Ical)=k+1 \\ \Ical \text{ lonely}}} (-1)^{k-1} (\# I_1-1)! \cdots (\# I_{k+1}-1)! \cdot G_{k \times k}(A_\Ical)^2 \\[0.2cm]
\label{eqn: general proof 1st contribution} & = \dfrac{(-1)^{n+1}}{12} \sum_{k=0}^r \sum_{\substack{\Ical \vdash [n + 1] \\ \ell(\Ical)=k+1 \\ \Ical \text{ lonely}}} (-1)^k (\# I_1-1)! \cdots (\# I_{k+1}-1)! \cdot G_{k \times k}(A_\Ical)^2.
\end{align}

We now turn to the second term on the right-hand side of \eqref{eqn: proof general case recursion of Euler chars}. Given $i \in [n]$ write $A(i)$ for the matrix:
\[ A(i) \colonequals \begin{bmatrix} \mbox{-----}a(i)\mbox{-----} \\ B \end{bmatrix}. \]
The contributions to $\chi_\orb(\DR_{1,n}^r(A(i)))$ are indexed by partitions $\Ical \vdash [n]$ of length $k+1$ for $k=0,\ldots,r$. For each such partition, the associated contraction satisfies:
\[ A(i)_\Ical = A_{\Ical_i} \]
where $\Ical_i \vdash [n\!+\!1]$ is the partition obtained by appending $n\!+\!1$ to the part of $\Ical$ containing $i \in [n]$. Note that $\Ical$ and $\Ical_i$ have the same length. In this way, we enumerate all the friendly partitions of $[n\!+\!1]$, and each such partition appears $(\# I_{k+1}\!-\!1)$ times, where without loss of generality $I_{k+1}$ is the part containing $n\!+\!1$. We conclude:
\begin{align}
\nonumber - \sum_{i=1}^n \chi_\orb(\DR_{1,n}^r(A(i))) & = - \sum_{i=1}^n \dfrac{(-1)^n}{12} \sum_{k=0}^r \sum_{\substack{\Ical \vdash [n] \\ \ell(\Ical)=k+1}} (-1)^k (\# I_1 - 1)! \cdots (\# I_{k+1} - 1)! \cdot G_{k \times k}(A(i)_\Ical)^2 \\[0.2cm]
\label{eqn: general proof 2nd contribution} & = \dfrac{(-1)^{n+1}}{12} \sum_{k=0}^r \sum_{\substack{\Ical \vdash [n+1] \\ \ell(\Ical)=k+1 \\ \Ical \text{ friendly}}} (-1)^k (\# I_1 - 1)! \cdots (\# I_{k+1} - 1)! \cdot G_{k \times k}(A_\Ical)^2.
\end{align}
Combining \eqref{eqn: general proof 1st contribution} and \eqref{eqn: general proof 2nd contribution} we obtain precisely the desired formula for $\chi_\orb(\DR_{1,n+1}^r(A))$. This completes the induction step.
\end{proof}

\subsection{Simplifying the leading term} \label{sec: leading term} We refer to the $k=r$ term in \eqref{eqn: higher rank formula} as the \textbf{leading term}:
\begin{equation} \label{eqn: leading term old} \dfrac{(-1)^{n+r}}{12} \sum_{\substack{\Ical \, \vdash [n] \\ \ell(\Ical)=r+1}} (\# I_1-1)! \cdots (\# I_{r+1}-1)! \cdot G_{r \times r}(A_\Ical)^2. \end{equation}
We now simplify the leading term, expressing it in terms of minors of the original matrix $A$ rather than its contractions $A_\Ical$.

Fix an $r \times n$ double ramification matrix $A$ and assume $n \geqslant r+1$ (otherwise the leading term vanishes). The $r \times r$ minors arise by selecting $r$ columns of $A$. Given a subset $I \subseteq [n]$ of size $r$ we let
\[ M_I(A) \]
denote the associated $r \times r$ minor.

\begin{proposition} \label{prop: leading term} The leading term \eqref{eqn: leading term old} is equal to:
\begin{equation} \label{eqn: leading term new} 
L_{1,n}^r(A) \colonequals \dfrac{(-1)^{n+r}}{12} \dfrac{(n-1)!}{(r+1)!} \sum_{I \in {[n] \choose r}} M_I(A)^2.
\end{equation}
\end{proposition}

\begin{proof}

We repeat the proof of \Cref{thm: higher rank DR}, replacing the old leading term \eqref{eqn: leading term old} by the new leading term \eqref{eqn: leading term new} in the formula. We adopt the same notation as before.

The base of the induction is straightforward. For the induction step, we note that $L_{1,n+1}^r(A)$ is $\GL_r(\Z)$-invariant, so we may reduce to matrices of the form \eqref{eqn: special form for DR matrix} and apply \Cref{thm: recursion higher rank} to obtain \eqref{eqn: proof general case recursion of Euler chars}. We saw in the proof of \Cref{thm: higher rank DR} that the leading terms on both sides of \eqref{eqn: proof general case recursion of Euler chars} are identified, hence we can focus exclusively on these. It remains to prove:
\begin{equation} \label{eqn: simplify leading term recursion} L_{1,n+1}^r(A) = a_{n+1}^2 \cdot L_{1,n}^{r-1}(B) - \sum_{i=1}^n L_{1,n}^r(A(i)). \end{equation}
The rest of the argument consists of algebraic manipulations. We say a subset
\[ I \in {[n\!+\!1] \choose r} \]
is \textbf{nouveau} if $n+1 \in I$ and \textbf{ancien} if $n+1 \not\in I$. In the nouveau case, we have
\[ 
M_I(A) = a_{n+1} \cdot M_{I \setminus \{n+1\}}(B).\smallskip
\]
Thus the first term on the right-hand side of \eqref{eqn: simplify leading term recursion} is equal to a sum of nouveau contributions, which we can write as:
\begin{equation} \label{eqn: simplification contribution 1} \dfrac{(-1)^{n+1+r}}{12} \dfrac{(n-1)!}{(r+1)!} (r+1) \sum_{\substack{I \in {[n+1] \choose r} \\ n+1 \in I}} M_I(A)^2. \end{equation}
The second term contains both nouveau and ancien contributions which we now disentangle. We can write:
\begin{equation} \label{eqn: simplification proof second term of RHS}
-\sum_{i=1}^n L_{1,n}^r(A(i))=\dfrac{(-1)^{n+1+r}}{12} \dfrac{(n-1)!}{(r+1)!} \sum_{i=1}^n \left(\sum_{\substack{I \in {[n] \choose r}\\ i \not\in I}} M_I(A(i))^2+\sum_{\substack{I \in {[n] \choose r}\\ i \in I}} M_I(A(i))^2 \right).
\end{equation}
For $i \in [n]$ and $I \in {[n] \choose r}$ we have
\[ 
M_I(A(i)) = \begin{cases} M_I(A) \qquad \qquad \qquad \qquad \qquad \qquad \qquad \ \text{if $i \not\in I$,} \\ M_I(A) + (-1)^{s(I,i)} M_{I \setminus \{i\} \cup \{n+1\}}(A) \qquad \text{if $i \in I$,} \end{cases}
\]
where $s(I,i) \colonequals \# \{ j \in I \colon j > i \}$. Then the sum in \eqref{eqn: simplification proof second term of RHS} over $I \not\ni i$ produces the following ancien contributions:
\begin{equation} \label{eqn: simplification contribution 2} \dfrac{(-1)^{n+1+r}}{12} \dfrac{(n-1)!}{(r+1)!} (n-r) \sum_{\substack{I \in {[n+1] \choose r} \\ n+1 \not\in I}} M_I(A)^2.\end{equation}
We now turn to the sum in \eqref{eqn: simplification proof second term of RHS} over $I \ni i$. Squaring $M_I(A(i))$ produces the mixed term
\[ 2 (-1)^{s(I,i)} \cdot M_I(A)\cdot M_{I \setminus \{i\} \cup \{n+1\}}(A).\]
Summing over $i \in [n]$ and setting $I^\prime \colonequals I\setminus\{i\}$ we can rewrite the sum of the mixed terms as follows
\begin{align*}
\sum_{i=1}^n\sum_{\substack{I \in {[n] \choose r}\\ i\in I}} 2(-1)^{s(I,i)} \cdot M_I(A) \cdot M_{I \setminus \{i\} \cup \{n+1\}}(A) & = \sum_{I^\prime \in{[n]\choose r-1}}\sum_{i\in[n]\setminus I^\prime} 2(-1)^{s(I,i)} \cdot M_{I^\prime \cup \{i\}}(A) \cdot M_{I^\prime \cup \{n+1\}}(A) \\[0.2cm]
& = 2 \sum_{I^\prime \in{[n]\choose r-1}} M_{I^\prime \cup \{n+1\}}(A) \sum_{i\in[n]\setminus I^\prime} (-1)^{s(I,i)} \cdot M_{I^\prime \cup \{i\}}(A) \\[0.2cm]
 &=-2\sum_{I^\prime \in{[n]\choose r-1}}M_{I^\prime \cup \{n+1\}}(A)^2 
\end{align*}
where the final equality follows from a basic property of double ramification matrices, similar to \Cref{lem: matrix lemma 2}. On the other hand, the square of $M_{I^\prime \cup \{n+1\}}(A)$ also appears in the summation once for every $i\in[n]\setminus I^\prime$, of which there are $n-(r-1)$. Assembling, we obtain
\begin{align*} \sum_{i=1}^n \sum_{\substack{I \in {[n] \choose r} \\ i \in I}} M_I(A(i))^2 & = \bigg( \sum_{i=1}^n \sum_{\substack{I \in {[n] \choose r} \\ i \in I}} M_I(A)^2 \bigg) + (-2+(n-(r-1)))\sum_{\substack{I \in {[n+1]\choose r} \\ n+1 \in I}} M_I(A)^2 \\[0.2cm]
& = r \sum_{\substack{I \in {[n+1] \choose r} \\ n+1 \not\in I}} M_I(A)^2 + (n-r-1)\sum_{\substack{I \in {[n+1]\choose r} \\ n+1 \in I}} M_I(A)^2
\end{align*}
so that the contribution is:
\begin{equation} \label{eqn: simplification contribution 3} \dfrac{(-1)^{n+1+r}}{12} \dfrac{(n-1)!}{(r+1)!} \bigg( r \sum_{\substack{I \in {[n+1] \choose r} \\ n+1 \not\in I}} M_I(A)^2 + (n-r-1)\sum_{\substack{I \in {[n+1]\choose r} \\ n+1 \in I}} M_I(A)^2 \bigg). \end{equation}
Combining \eqref{eqn: simplification contribution 1}, \eqref{eqn: simplification contribution 2}, and \eqref{eqn: simplification contribution 3} we obtain the desired identity \eqref{eqn: simplify leading term recursion}.
\end{proof}

\subsection{Comparing Theorems~\ref{thm: Euler char rank one}~and~\ref{thm: higher rank DR}}  \label{sec: comparing theorems X and Y} When $r=1$ we can directly match the formula appearing in \Cref{thm: higher rank DR} with the considerably simpler formula appearing in \Cref{thm: Euler char rank one}. We require the following:

\begin{lemma} Fix $a=(a_1,\ldots,a_n) \in \Z^n$ with $\Sigma_{i=1}^n a_i =0$. For each $m \in \{1,\ldots,n-1\}$ we have
\[ \sum_{I \in {[n] \choose m}} a_I^2 = {n-2 \choose m-1} \sum_{i=1}^n a_i^2 \]
where $a_I \colonequals \sum_{i \in I} a_i$.
\end{lemma}

\begin{proof} The left-hand side is a homogeneous quadratic symmetric polynomial in $a_1,\ldots,a_n$ and hence can be written in terms of power sums \cite[I (2.12)]{Macdonald} as
\[ \lambda_1 \cdot \left( \Sigma_{i=1}^n a_i \right)^2 + \lambda_2 \cdot \Sigma_{i=1}^n a_i^2 \]
for some $\lambda_1,\lambda_2 \in \Q$. The first term vanishes, and to determine $\lambda_2$ it suffices to evaluate at a single vector. Take $a=(1,-1,0,\ldots,0)$. Then $a_I=0$ unless $I$ and $[n] \setminus I$ separate $1$ and $2$. Enumerating separately the cases $1 \in I$ and $2 \in I$ we obtain:
\[ \sum_{I \in {[n] \choose m}} a_I^2 = 2 {n-2 \choose m-1}.\]
On the other hand $\Sigma_{i=1}^n a_i^2=2$. We conclude that $\lambda_2 = {n-2 \choose m-1}$ as required.
\end{proof}

Now consider the formula in \Cref{thm: higher rank DR}. Since $r=1$ we sum over $k=0$ and $k=1$. For $k=0$ we have a single partition of length $1$, and by convention $G_{0 \times 0}(A_\Ical)=1$. The contribution is:
\begin{equation} \label{eqn: rk 1 simplify 1st term} \dfrac{(-1)^n (n-1)!}{12}.\end{equation}
For $k=1$ we sum over partitions $\Ical=\{I_1,I_2\}$ of length $2$. This is equal to half the sum over subsets $I \subseteq [n]$ of size $m \in \{1,\ldots,n-1\}$. Each subset leads to a $1 \times 2$ matrix giving:
\[ G_{1 \times 1}(A_\Ical) = a_I.\]
The contribution is thus:
\begin{align} \nonumber & \dfrac{(-1)^{n}}{12} \cdot \dfrac{(-1)^1}{2} \cdot \sum_{m=1}^{n-1} \bigg( (m-1)!(n-m-1)! \sum_{I \in {[n] \choose m}} a_I^2 \bigg) \\[0.2cm]
\nonumber = \ \ & \dfrac{(-1)^{n+1}}{24} \left( \sum_{m=1}^{n-1} (m-1)!(n-m-1)! {n-2 \choose m-1} \right) \sum_{i=1}^n a_i^2 \\[0.2cm]
\label{eqn: rk 1 simplify 2nd term} = \ \ & \dfrac{(-1)^{n+1}(n-1)!}{24} \sum_{i=1}^n a_i^2. \end{align}
Combining \eqref{eqn: rk 1 simplify 1st term} and \eqref{eqn: rk 1 simplify 2nd term} we obtain the formula in \Cref{thm: Euler char rank one}.

\footnotesize
\bibliographystyle{alpha}
\bibliography{Bibliography.bib}\medskip

\noindent Luca Battistella, University of Bologna, \href{mailto:luca.battistella2@unibo.it}{luca.battistella2@unibo.it}.

\noindent Navid Nabijou, Queen Mary University of London, \href{mailto:n.nabijou@qmul.ac.uk}{n.nabijou@qmul.ac.uk}.

\end{document}